\documentclass[preprint,10pt]{article}

\usepackage{authblk}

\usepackage{amsmath}
\usepackage{amssymb}
\usepackage{amsthm}
\usepackage{enumerate}
\usepackage[mathscr]{eucal}

\theoremstyle{plain}
\newtheorem{theorem}{Theorem}[section]
\newtheorem{prop}[theorem]{Proposition}
\newtheorem{problem}[theorem]{Problem}
\theoremstyle{definition}
\newtheorem{definition}{Definition}[section]

\theoremstyle{remark}

\numberwithin{equation}{section}

\usepackage{color}
\usepackage{graphicx}

\def\relsys#1{\mathbf {#1}}
\def\rel#1#2{R_{\mathbf{#1}}^{#2}}
\def\ext#1#2{X_{\mathbf{#1}}^{#2}}
\def\Rel{\mathop{\mathrm{Rel}}\nolimits}

\DeclareMathOperator{\Aut}{Aut}
\DeclareMathOperator{\Inv}{Inv}

\DeclareMathOperator{\rc}{rc}
\DeclareMathOperator{\lc}{lc}

\def\Age{\mathop{\mathrm{Age}}\nolimits}
\def\Forb{\mathop{\mathrm{Forb_h}}\nolimits}
\def\F{{\mathcal F}}

\def\K{{\mathcal K}}

\def\Fraisse{Fra\"{\i}ss\' e}

\newcommand{\UNION}{\bigcup}

\begin{document}

\title{Complexities of relational structures}

\author{David Hartman}
\author{Jan Hubi\v{c}ka}
\author{Jaroslav Ne\v{s}et\v{r}il}

\affil{Computer Science Institute of Charles University in Prague\footnote{The Computer Science Institute of Charles University (IUUK) is supported by grant ERC-CZ LL-1201 of the Czech Ministry of Education and CE-ITI P202/12/G061 of GA\v CR}, Malostransk\' e n\' am\v est\' i 25, 118~00, Prague 1, Czech Republic}

\date{}

\maketitle

\begin{abstract}
The relational complexity, introduced by G. Cherlin, G. Martin, and D. Saracino,
is a measure of ultrahomogeneity of a relational structure. It provides
an information on minimal arity of additional invariant relations needed to turn
given structure into an ultrahomogeneous one.  The original motivation was
group theory. This work focuses more on structures and provides an alternative
approach. Our study is motivated by related concept of lift complexity studied by
Hubi\v{c}ka and Ne\v{s}et\v{r}il.
\end{abstract}

\section{Introduction}
A {\em relational structure} (or simply {\em structure}) $\relsys{A}$ is a pair
$(A,(\rel{A}{i}:i\in I))$, where $\rel{A}{i}\subseteq A^{\delta_i}$ (i.e.,
$\rel{A}{i}$ is a $\delta_i$-ary relation on $A$). The family $(\delta_i: i\in
I)$ is called the {\em type} $\Delta$. The type is assumed to be fixed
and understood from context throughout this paper. The class of all (countable) relational structures of
type $\Delta$ will be denoted by $\Rel(\Delta)$.  If the set $A$ is finite we
call $\relsys A$ a {\em finite structure}.  We consider only countable or
finite structures. 

We see relational structures as a generalization of graphs and digraphs, and
adopt standard graph theoretic terms (such as isomorphism, homomorphism or
connected structures). We also use standard model theoretic notation, see
\cite{Hodges:1993}.

\begin{figure}[ht!]
\centerline{\includegraphics[width=6cm]{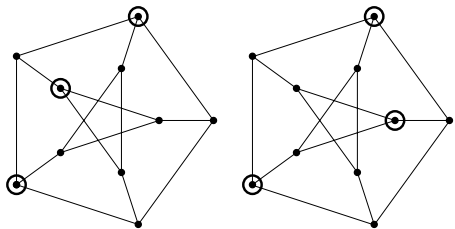}}
\caption{Two types of independent sets of size 3 in the Petersen graph.}
\label{fig:petersenhom}
\end{figure}
A relational structure $\mathbf{A}$ is called {\em ultrahomogeneous} (or simply {\em homogeneous}) if every isomorphism between two induced finite substructures of $\relsys{A}$ can be extended to an automorphism of $\relsys{A}$. This property represents high degree of symmetry. Most of the usual properties of symmetry are implied by ultrahomogeneity --- for example, in the finite case, the vertex-transitivity, the edge-transitivity or even the distance transitivity are implied by ultrahomogeneity. On the other hand usual examples of highly symmetric graphs are not necessarily ultrahomogeneous. The Petersen graph is a well known example of a symmetric and non-ulrahomogeneous graph; it contains two independent sets of size three (depicted in Figure \ref{fig:petersenhom}) that cannot be mapped one to another by an automorphism. 

There exists a long standing classification program of ultrahomogeneous structures, see for example~\cite{LachlanWoodrow:1980,Cherlin:1998}. The results are nontrivial even in the seemingly simple case of finite undirected graphs without loops. Catalog of those graphs was given by Gardiner~\cite{Gardiner:1976}: a finite graph is ultrahomogeneous if and only if it is isomorphic to one of the following: (1) disjoint union of finitely many copies of a complete graph $K_r$, (2) complete multipartite graph (complement of (1)), (3) a $5$-cycle, and (4) a line graph $L(K_{3,3})$ of the complete bipartite graph $K_{3,3}$.

By the above classification cycles are ultrahomogeneous up to size $5$ ($C_3$ is a complete graph, $C_4$ is a complete bipartite graph, $C_5$ is a sporadic case). Similarly as the Petersen graph, $C_6$ has two different independent sets of size $2$ (a pair of vertices of distance 2 and a pair of vertices of distance 3). $C_6$ can be turned into an ultrahomogeneous structure by introducing another type of edges, say red edges, and by connecting every pair of vertices of distance 2 by a red edge. Resulting structure is ultrahomogeneous and we consider such extended structure to be a {\em homogenization} of structure. Cherlin~\cite{Cherlin:2000}, when explaining more general theory of ultrahomogeneous structures~\cite{KnightLachlan:1985}, notices that such a process is more general --- classifications of ultrahomogeneous structures usually contain sporadic ultrahomogeneous structures that can be considered as members of regular ultrahomogeneous families in an extended language. Althrough $C_5$ is a sporadic case of ultrahomogeneous graph, it is a part of a family of graph cycles that are all ultrahomogeneous as metric graphs. Homogenization is also important concept of constructing universal structures \cite{Covington:1990, BubickaNesetril:toapear}.

\section{Complexity of relational structures}

We introduce our two notions of complexity; the relational complexity and the lift complexity. First, however, we take time to review the process of adding new types of relations into a structure more formally.

Let $\Delta'=(\delta'_i:i\in I')$ be a type containing type $\Delta$. (By this
we mean $I\subseteq I'$ and $\delta'_i=\delta_i$ for $i\in I$.) Then every
structure $\relsys{X}\in \Rel(\Delta')$ may be viewed as a structure
$\relsys{A}=(A,(\rel{A}{i}: i\in I))\in \Rel(\Delta)$ together with some
additional relations for $i\in I'\setminus I$. To make this more
explicit, these additional relations will be denoted by $\ext{X}{i}, i\in
I'\setminus I$. Thus a structure $\relsys{X}\in \Rel(\Delta')$ will be written
as $$\relsys{X}=(A,(\rel{A}{i}:i\in I),(\ext{X}{i}:i\in I'\setminus I)),$$ and,
abusing notation, more briefly as
$$\relsys{X}=(\relsys{A},\ext{X}{1},\ext{X}{2},\ldots, \ext{X}{N}).$$

We call $\relsys{X}$ a {\em lift} of $\relsys{A}$ and $\relsys{A}$ is called
the {\em shadow} of $\relsys{X}$. In this sense the class $\Rel(\Delta')$ is
the class of all lifts of $\Rel(\Delta)$.  Conversely, $\Rel(\Delta)$ is the
class of all shadows of $\Rel(\Delta')$. If all extended relations are unary,
the lift is called {\em monadic}.  In the context of monadic lifts, the {\em
color} of vertex $v$ representing unary relation is the set $\{i;(v)\in \ext{X}{i}\}$.
Note that a lift is also in the model-theoretic setting called an {\em
expansion} (as we are expanding our relational language) and a shadow a {\em reduct} (as we are reducing it).  Unless stated explicitely, we shall
use letters $\relsys{A}, \relsys{B}, \relsys{C}, \ldots$ for shadows (in
$\Rel(\Delta)$) and letters $\relsys{X}, \relsys{Y}, \relsys{Z}$ for lifts (in
$\Rel(\Delta'))$.

The {\em lift complexity}, $\lc(\relsys{A})$, of
relational structure $\relsys{A}$ is the least $k$ such that there exists a lift
$\relsys{X}=(\relsys{A},\ext{X}{1},\ext{X}{2},\ldots, \ext{X}{N})$ of $\relsys{A}$ that is
ultrahomogeneous and all the relations $\ext{X}{1},\ext{X}{2},\ldots,
\ext{X}{N}$ have arity at most $k$.

Let $\relsys{A}$ be a relational structure and let $\Aut(\relsys{A})$ be the automorphism group of $\relsys{A}$. A $k$-ary relation $\rho \subseteq A^k$ is an \emph{invariant} of $\Aut(\relsys{A})$ if $(\alpha(x_1), \ldots, \alpha(x_k)) \in \rho$ for all $\alpha \in \Aut(\relsys{A})$ and all $(x_1, \ldots, x_k) \in \rho$. Let $\Inv_k(\relsys{A})$ denote the set of all $k$-ary invariants of $\Aut(\relsys{A})$ and let $\Inv(\relsys{A}) = \UNION_{k \ge 1} \Inv_k(\relsys{A})$, $\Inv_{\leq k}(\relsys{A})=\UNION_{1\leq k'\leq k} \Inv_{k'}(\relsys{A})$. It easily follows that lift $(A, (\rel{A}{i}:i\in I),\Inv(\relsys{A}))$ (possibly of infinite type) is an ultrahomogeneous structure for every structure $\relsys{A} = (A,(\rel{A}{i}:i\in I))$. For a  structure $\relsys{A}$ the \emph{relational complexity $\rc(\relsys{A})$} of $\relsys{A}$ is the least $k$ such that $(A, (\rel{A}{i}:i\in I), \Inv_{\leq_k}(\relsys{A}))$ is ultrahomogeneous, if such a $k$ exist. If no such $k$ exists, we say that the relational complexity of $\relsys{A}$ is not finite and write $\rc(\relsys{A}) = \infty$.  

\section{Basic properties}
As a warmup, we state few basic observations about relational and lift complexities. 
It easily follows from the definition, that complement of ultrahomogeneous structure is also ultrahomogeneous. The same holds for both of our notions of complexity.
\begin{prop}
Relational and lift complexity is closed under complementation.
\end{prop}

The relational complexity is determined similarly to a lift complexity moreover considering further restrictions on adding relations given by automorphism group. This immediately leads to the following simple fact.
\begin{prop}
For every structure $\relsys{A}$, $\lc(\relsys{A})\leq \rc(\relsys{A})$.
\end{prop}

The relational complexity is interesting even for finite structures, while the lift complexity is trivially $1$ for finite structures that are not ultrahomogeneous.
\begin{prop}
\label{prop:upfinite}
Let $\relsys{A}$ be finite relational structure. Then $\lc(\relsys{A})\leq 1$ and $\rc(\relsys{A})\leq |A|-1$.
\end{prop}
\begin{proof}
For every finite $\relsys{A}$ an ultrahomogeneous lift can be created by adding unique
unary relation to every vertex. This give $\lc(\relsys{A})\leq 1$.

The unary relations may not be invariant.  It can be however easily seen that
by adding all invariant relations (i.e. those having arity at most $|A|$) one
always obtain an ultrahomogeneous structure. The relational structure of arity
$|A|$ do not however contribute into the homogenization of the structure giving
the bound of $|A|-1$ on the relational complexity of finite structure.
\end{proof}

%In the rest of the paper we will give number of examples of finite graphs with
%non-trivial relational complexities. 

The following observation 
allows us to restrict our attention to connected structures.

\begin{prop}
\label{prop:disjoint}
Let $k \geq 2$ be finite and let $\relsys{A}$ be non-ultrahomogeneous relational structure with connected components  $\relsys{A}_1, \relsys{A}_2, \ldots ,\relsys{A}_k$. Then 
\begin{enumerate}
\item \label{item1} $\lc(\relsys{A}) = \max\{1,\lc(\relsys{A}_1),\lc(\relsys{A}_2),\ldots \lc(\relsys{A}_k)\};$
\item 
if there is a pair of two mutually isomorphic structures $\relsys{A}_i$ and $\relsys{A}_j$, $i\neq j$, such that $rc(\relsys{A}_i + \relsys{A}_j)>1$, then
$$\rc(\relsys{A}) = \max\{2,\rc(\relsys{A}_1),\rc(\relsys{A}_2),\ldots \rc(\relsys{A}_k)\},$$
otherwise
$$\rc(\relsys{A}) = \max\{1,\rc(\relsys{A}_1),\rc(\relsys{A}_2),\ldots \rc(\relsys{A}_k)\}.$$
\end{enumerate}
\end{prop}
(Here $\relsys{A} + \relsys{B}$ denote the disjoint union of $\relsys{A}$ and $\relsys{B}$.)
\begin{proof}
The ultrahomogeneous lift of the structure $\relsys{A}$ can always be created as a disjoint
union of ultrahomogeneous lifts of $\relsys{A}_1, \relsys{A}_2, \ldots ,\relsys{A}_k$ with $k$
additional unary relations distinguishing individual components. With this construction
the lift complexity is increased to at least $1$. This finishes proof of (1).

To show (2) we need to add only invariant unary and binary relations depending on isomorphism type of components. First we add unary relations
classifying vertices by an isomorphism type of a connected component they belong to.
If there is no pair of mutually isomorphic components $\relsys{A}_i,\relsys{A}_j$ such
that $\rc(\relsys{A}_i + \relsys{A}_j)>1$, then the lift is extended by all necessary relations that make each component ultrahomogeneous which 
creates an ultrahomogeneous lift. 

For every $\relsys{A}_i$ along with one or more components $\relsys{A}_{j_1}$, $\relsys{A}_{j_2}$,
\ldots $\relsys{A}_{j_n}$ that are all distinct and isomorphic to $\relsys{A}_i$, we add
two extra binary relations; first relation used to join
all pairs of vertices within the same component, while second joins all pairs
of vertices within two different components.

These additional relations prevents partial isomorphisms exchanging vertices in between 
individual connected components leading to an ultrahomogeneous structure after applying all remaining relations like above.
\end{proof}

\subsection{Graphs of complexity 1} 
Our notion of relational complexity of structures is derived from the notion of relational
complexity of groups in \cite{Cherlin:2000}.   For every structure
$\relsys{A}$, the relational complexity of $\Aut(\relsys{A})$ (in sense of \cite{Cherlin:2000}) corresponds to 
$\rc(\relsys{A})$ with the exception of $\rc(\relsys{A})$ being smaller than
the maximal arity of relation in $\relsys{A}$. Our notion of relational
complexity ignore arities of relations already present in $\relsys{A}$. For
instance, the relational complexity of cyclic group on $n\geq 4$ elements in the sense of
\cite{Cherlin:2000} is 2, while $\rc(C_4)$ and $\rc(C_5)$ is 0.
The main motivation for our definition is to get
a finer information on structures with small complexity. We explore these small
complexity classes now.
To simplify our presentation, we will restrict our attention to graphs. 
(Many of our observations generalize into relational structures.)

The class of graphs of relational or lift complexity $0$ is well studied; those
are the ultrahomogeneous graphs. We now consider the class of graphs of relational complexity $1$ and the
class of graphs with lift complexity $1$. 
Both these clases are closely related to an established notion
of $n$-graph (see \cite{Rose:2011} for a recent review of the topic).

\begin{definition}
For $n$ a positive integer, an {\em $c$-colored $n$-graph} is a graph on $n$
pairwise disjoint sets of vertices $V_1, V_2, \ldots, V_n$ (called {\em parts})
each of which is an ordinary countable graph, with $c$ edge-types between pairs
of parts (cross edges).
\end{definition}
Isomorphisms of $n$-graphs do not permit exchanging vertices of individual parts~\cite{Rose:2011}.
We will consider $2$-colored $n$-graphs to correspond to graphs where one of
the types of cross edges is an edge and the other type is non-edge.

The unary relations forming an ultrahomogeneous lift (equivalently
seen as a vertex coloring) splits the vertex set of a graph into a finite
partition. In analogy with $n$-graphs  we call classes of this partition {\em parts}. 

The following observation describe the structure of graphs of complexity 1.

\begin{prop}\label{prop:rc1} Let $\relsys{G}$ be a graph with $\rc(\relsys{G}) = 1$ or $\lc(\relsys{G})=1$ and let $V_1,V_2,\ldots, V_k$ be its parts. Then the following holds.
\begin{enumerate}
\item The subgraph induced by each part is an ultrahomogeneous graph.
\item $\relsys{G}$ corresponds to an ultrahomogeneous $2$-colored $k$-graph with partitions $V_1,V_2,\ldots, V_k$.
\item The subgraph induced by each pair of parts corresponds to an ultrahomogeneous $2$-colored $2$-graph.
\end{enumerate}
\end{prop}
\begin{proof}
The automorphisms of $\relsys{G}$ consists of arbitrary combinations of automorphisms of
subgraphs induced by the individual parts, consequently, each of
the subgraphs must already be ultrahomogeneous giving (1).

(2) and (3) follows directly from the ultrahomogeneity.
\end{proof}

The structural condition given by Proposition \ref{prop:rc1} and the
classification of ultrahomogeneous graphs allows construction of number of
interesting examples. Generally $c$-colored $n$-graphs are only partially classified. The classification is complete for the class of {\em $n$-edge-colored
bipartite} graphs that is a special subclass of an $n$-colored $2$-graphs where every partition is an independent set.

\begin{theorem}[\cite{JenkinsonSeidelTruss:2012}]
If $\relsys{G}$ is a countable ultrahomogeneous $n$-edge-colored bipartite graph such that
$n$ is finite and all of the $n$ types of edges are used in $\relsys{G}$, then one of
the following holds:
\begin{itemize}
  \item $n=1$ and all edges are the same color;
  \item $n=2$ and edges of one color are a perfect matching, and edges of the other color are its complement;
  \item $n\geq 2$ and $\relsys{G}$ is a generic bipartite $n$-edge-colored graph.
\end{itemize}
\end{theorem}
\begin{figure}[ht!]
\centerline{\includegraphics[width=6cm]{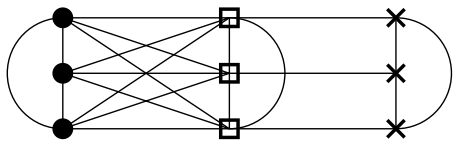}}
\caption{Graph of relational complexity 1 consisting of 3 isomorphic ultrahomogeneous graph ($K_3$). The parts
are depicted by different vertex markings.}
\label{fig:rc1-a}
\end{figure}

In the case of graphs ($n\leq 2$), we thus have only three types of bipartite ultrahomogeneous graphs (up to complements); complete or empty bipartite graphs,
matchings and the generic bipartite graph.
In Figure \ref{fig:rc1-a} we show a graph of relational complexity 1 that consist of 3 isomorphic ultrahomogeneous subgraphs.

Observe that not every graph built from finitely many of isomorphic copies of ultrahomogeneous graphs necessarily have relational complexity $1$. In the example given it is the use of different bipartite graphs to connect each pair of subgraphs ensuring the fact
that each of the subgraphs forms an independent cluster. Proposition \ref{prop:rc1} can however
be reversed for lift complexity:
\begin{prop}
Every infinite graph $G$ with $n$ parts ($n$ finite) such that $\lc(G)=1$ correspond to a ultrahomogeneous $2$-colored $n$-graph.
\end{prop}

\begin{figure}[ht!]
\centerline{\includegraphics[width=6cm]{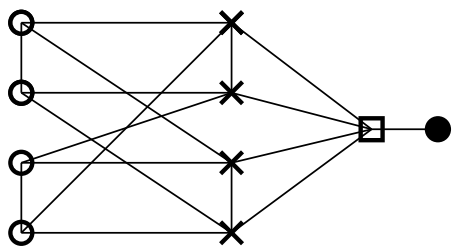}}
\caption{Graph of relational complexity 1 using a sporadic $2$-colored $2$-graph.}
\label{fig:rc1-b}
\end{figure}
The classification of ultrahomogeneous $2$-colored $n$-graphs is still an open problem, see
\cite{Rose:2011} for partial results.  The graph depicted in Figure \ref{fig:rc1-b} shows a graph of
relational complexity $1$ that is constructed from 4 ultrahomogeneous subgraphs
($K_1$, $K_1$, $K_2 + K_2$ and $K_2 + K_2$) using an sporadic example
of ultrahomogeneous 2-colored 2-graph joining $K_2 + K_2$ and $K_2 + K_2$
(c.f. \cite{Rose:2011}).

A special case of graphs of lift complexity 1 can be constructed from a given
finite graph $\relsys{G}$ by replacing every vertex by an ultrahomogeneous
graph and every edge by an ultrahomogeneous $2$-colored $2$-graph. Such
construction is considered in \cite{BubickaNesetril:toapear} where homomorphism
dual $\relsys{D}$ of graph $\relsys{T}$ is turned into a countable graph
$\relsys{U}$ that is embedding universal for the class of all graphs not
containing homomorphic image of $\relsys{T}$. Here vertices of $\relsys{D}$ are
replaced by infinite discrete graphs and edges of $\relsys{D}$ by random
bipartite graphs.  It is well known result \cite{NesetrilTardif:2000} that
homomorphism duals exists only when $\relsys{T}$ is an (relational) tree.  For
the case of universal graphs with lift complexity $1$, the construction can be
extended also for structures constructed from trees by replacing every edge by
an arbitrary irreducible structure.   When the same construction is applied on
any rigid graph  $\relsys{D}$ (and this is the case of all cores of graph
duals), the expanded graph has also relational complexity $1$.
We will further study complexity of infinite graphs in a greater detail in Section \ref{sec:infinite}.

%To see that for infinite graphs the lift and relational complexity does not
%need to agree consider the infinite random bipartite graph.  Its lift
%complexity is 1 (the partitions can be distinguished by different relations)
%while its relational complexity is 2. Since there are automorphisms exchanging the
%partitions, it is necessary to use 2 binary relations; one joining vertices
%within same partition, while the second joining vertices across the partitions.

%\todo{{\bf ?:} Tohle nevim, jestli nedat nejak do toho dukazu nebo jestli dukaz delat?}
%For situation where clusters in pair are complete or empty, only connections are full and empty bipartite graphs or a perfect \todo{?} matching. However when there are different clusters some sporadic types of graphs can emerge as shown in following example.
%
%\begin{example}[Ruzickuv sporadak ...] \todo{{\bf Honzo:} zkusil bys to navrhnout. Meli by to byt ten sporadic case z leva a z prava postupne s pripojenymi $K_4$ a $K_5$} 
%\end{example}
%We close our discussion of relational complexity 1 with the following simple proposition.
%\begin{prop}\label{prop:rc1-trees}
%All finite (relational) trees have relational complexity 1
%\end{prop}
%\begin{proof}
%This fact follows from well known algorithm for testing of isomorphism in between
%trees (see i.e. \cite{NesetrilMatousek}) works by assigning a vertices labels that are invariant under automorphism
%and turn the tree into a ultrahomogeneous structure.
%\end{proof}

\subsection{Graphs of complexity 2} 

Given the difficulties of characterizing even graphs of complexity
$1$, it is not really reasonable to expect simple characterization of graphs of
complexity $2$. We can however show several interesting classes of such
graphs. %The first example is a class that almost falls into previous case of graphs 
%with relational complexity $1$. 

%Keep in mind well known algorithm for testing of isomorphism in between trees (see e.g. \cite{MatousekNesetril:1998}) we can infer that an upper bound for relational complexity for trees could be $1$. There are however simple examples showing that it is not the case -- assume three paths with $3$ vertices where centres of two are identified each one with different leaves of the third. Such a tree has relational complexity $2$. We can see that those are in fact the only examples.

Given graph $\relsys{G}$, the {\em graph metric} is a function measuring path distance of
two vertices.  We call graph {\em metrically ultrahomogeneous} if and only if it is
ultrahomogeneous as a metric space with its graph metric.  It is not difficult to
see, that relations measuring distance of vertices (i.e.  relations connecting
vertices of distance $n$) are all invariant binary relations.  We immediately get
the following observation.
\begin{prop}
All metrically ultrahomogeneous graphs have relational and lift complexity at
most 2.
\end{prop}

The metrically ultrahomogeneous graphs was studied by Cherlin leading to partial catalog
\cite{Cherlin:2011}. These include special bipartite graphs,
tree-like graphs, antipodal graphs and number of other examples.

On example of a connected graph of relational complexity 2 that is not metrically ultrahomogeneous can be constructed
by aid of Proposition \ref{prop:disjoint}. Take complement of graph created as
$C_5 + C_5$.  This graph has relational complexity 2, but it is not metrically
ultrahomogeneous.  Consider an function mapping edge within complement of $C_5$ into and
edge joining the two subgraphs.  Additional examples can be easily produced by
aid of the following two propositions.

\begin{prop}\label{prop:rc1-trees}
Finite (graph) trees have relational complexity at most $2$.
\end{prop}
\begin{proof}
We prove this by induction on the diameter the tree $\relsys{T}$. Again we will consider
the unary relations to be vertex colorings.
To carry the induction, we prove a stronger result: all finite vertex colored trees have relational complexity at most $2$.

The claim trivially holds for colored tree consisting of a single vertex or an edge.

Now assume that the claim works for trees of diameter up to $k$.
 Fix colored tree $\relsys{T}$ of diameter $k+2$ and consider a tree $\relsys{T}'$ constructed from $\relsys{T}$ by
\begin{enumerate}
 \item removing all leaves; and
 \item changing color of every new leaf vertex $v$ to an unique representation of the isomorphism type of a rooted tree induced by $\relsys{T}$ on $v$ and its sons.
\end{enumerate}
$\relsys{T}'$ is is a vertex colored tree with diameter $k$. By the induction hypothesis $\rc(\relsys{T}')\leq 2$ and thus there exists a ultrahomogeneous lift $\relsys{X}'$ of $\relsys{T}'$ adding only invariant unary and binary relations. To make our presentation easier, we also consider the colors introduced in construction of $\relsys{T}'$ to be extended unary relations of $\relsys{X}'$ and edges of $\relsys{T}'$ to be mirrored in an extended binary relation of $\relsys{X}'$.

  We extend $\relsys{X}'$ to a ultrahomogeneous lift $\relsys{X}$ of $\relsys{T}$. Once again we extend our language (the type of $\relsys{X}$) in the following way:
\begin{enumerate}
\item for every unary relation $\ext{X'}{i}$ we add new unary relation $\ext{X}{u(i)}$ and new binary relation $\ext{X}{b(i)}$; and
\item for every binary relation $\ext{X'}{i}$ we add new binary relation $\ext{X}{b(i)}$.
\end{enumerate}
$\relsys{X}$ is a lift of $\relsys{T}$ such that $\ext{X}{i}=\ext{X'}{i}$ for all relations used by $\relsys{X}'$.  We use the newly introduced relations in the following way.
\begin{enumerate}
\item for every leaf $v\in T$ with father $v'$ such that $(v')\in \ext{X'}{i}$ we also put $(v)\in \ext{X}{u(i)}$;
\item for every pair of distinct leaves $v_1,v_2 \in T$ sharing father $v'$ such that $(v')\in \ext{X'}{i}$ we also put $(v_1,v_2)\in \ext{X}{b(i)}$ and $(v_2,v_1)\in \ext{X}{b(i)}$;
\item for every pair of distinct leaves $v_1,v_2 \in T$ with distinct fathers $v'_1,v'_2$ (respectively) such that $(v'_1,v'_2)\in \ext{X'}{i}$ we also put $(v_1,v_2)\in \ext{X}{b(i)}$.
\end{enumerate}
The ultrahomogeneity of $\relsys{X}'$ follows from the fact that automorphisms
of $\relsys{T}$ must map leaves to leaves and non-leaves to non-leaves.  The automorphism group of
$\relsys{T}$ acting on non-leaf vertices of $\relsys{T}$ is precisely the
automorphism group of $\relsys{X}'$. Finally every automorphism can map a leaf vertex
$v$ to a leaf vertex $v'$ if and only if they have same color and it can map the father of
vertex $v$ to the father of vertex $v'$.

\end{proof}

Because the graphs with relational complexity $2$ are closed under
complementation, the Proposition \ref{prop:disjoint} give an iterative way to
construct non-trivial examples of such graphs. Consider the following special case. {\em Cograph} is a graph not
containing an induced path on $4$ vertices. It is well known that all cographs
can be generated from the single-vertex graph $K_1$ by complementation and
disjoint union. We immediately get
\begin{prop}
\label{prop:cographs}
Finite cographs have relational complexity at most 2.
\end{prop}

%\todo{{\bf David:} tohle me prekvapuje hlavne proto, ze Covingtonova stavi ten univerzalni cograph (viz konec clanku) a potrebuje ternarni relaci. Jde to s binarni? A nebo nekde slozitost poskoci s limitou? A nebo jsme neco prehlidli?} 

\section{Finite graphs with large relational complexity} 
\label{sec:large}
It is not difficult to construct examples of finite graphs with large
relational complexity showing that the relational complexity of graphs is not bounded. 

Consider a permutation group $\Gamma$ acting on $n$
elements that is $k$-transitive but not $(k+1)$-transitive.  (If $\Gamma$ is
the alternating group on $n$ elements, then $k=n-2$.)  Now construct graph $\relsys{G}_\Gamma$ with the following vertices:
\begin{enumerate}
\item $n$ control vertices $v_1,v_2,\ldots v_n$;
\item for every permutation $p\in \Gamma$, $n+1$ additional vertices $v^p_1,v^p_2,\ldots , v^p_{n+1}$;
\end{enumerate}
and the following edges:
\begin{enumerate}
 \item $\{v^p_a,v^p_{a+1}\}$ for every $p\in \Gamma$ and $1\leq a\leq n$;
 \item $\{v^p_i, v_j\}$ if and only if the permutation $p\in \Gamma$ sends $i-th$ element to $j-th$ element.
\end{enumerate}
The graph $\relsys{G}_\Gamma$ thus consist of control vertices and paths representing individual permutations connected
to the control vertices by pairings. By the construction, automorphism of $\relsys{G}_\Gamma$ sends the control vertices to control vertices
(these are the only vertices of large degree).  Similarly the automorphism must send a path
$v^p_1,v^p_2,\ldots , v^p_{n+1}$ corresponding to some permutation $p\in \Gamma$ to another
path $v^{p'}_1,v^{p'}_2,\ldots , v^{p'}_{n+1}$ corresponding to another permutation $p'\in \Gamma$.
It easily follows that the automorphism group of $\relsys{G}_\Gamma$
acting on the control vertices is precisely $\Gamma$.  From transitivity of
$\Gamma$ it is necessary to use at least $(k+1)$-ary relations to homogenize it, giving
$\rc(\relsys{G}_\Gamma)> k$.
We shall remark that this construction in fact works for all groups with large relational complexity
in the sense of \cite{Cherlin:2000}. 

Identifying less artificial families of graphs
with large relational complexity is however challenging. We outline two examples given by~\cite{Cherlin:2000}.
%The first example, originally studied in~\cite{CherlinMartinSaracino:1996}, is defined as follows.
\begin{enumerate}
\item 
 The {\em Johnson's graph} $J_{n,k}$ is the graph whose correspond to the $k$-subsets of $n$ element set. Two vertices are adjacent when two corresponding sets meets in exactly $k - 1$ elements. \cite{CherlinMartinSaracino:1996} give a bound on relational complexity of Johnson's graphs: $\rc(J_{n,k}) \leq 2[\log_2 k]$. Equality is achieved when $n \geq 2\log_2k + 2$. 

\item
The {\em Knesser graphs} $KG_{n,k}$ is the graph whose vertices correspond to the  $k$-subsets of a set of $n$ elements, and where two vertices are adjacent if and only if the two corresponding sets are disjoint. For  $n \geq 2k$ the relational complexity is $\rc(KG_{n,k}) = 2[\log_2 k]$.

In particular, the relational complexity of the Petersen graph is $3$ and thus the problem shown in Figure~\ref{fig:petersenhom} is in fact
the only obstacle its ultrahomogeneity.
\end{enumerate}

It seems that the upper bound of relational complexity given by
Proposition~\ref{prop:upfinite} is far from reality, since the relational
complexity of all examples grows sublogarithmically in their size.  It is
reasonable to ask for extremal examples of graphs with high relational
complexity or for better bounds.

\begin{problem} Estimate $f(n)=\max_{|G|=n} rc(G)$.
\end{problem}

%\todo{Priklad se Steinerem.} \\
%\todo{{\bf David:} Dokazeme ze slozitost je maximalne logaritmicka a nebo to dame jako problem?}
%\todo{{\bf Honza:} Podle me je slozitost pro velikost grafu u Knesseru a Johnsonu horsi. U tech steineru by mohla byt lepsi, tj. vyssi, ale nemam to jeste vycislene} \\

\section{Complexity of infinite structures}
\label{sec:infinite}

We discussed primarily the relational complexity of finite structures. Now we turn our attention to the infinite case.  First we briefly outline
the \Fraisse{} Theorem that can be seen as a ``zero instance'' of problems considered in this section.

Let $\K$ be a class of finite structures, by $\Age(\K)$ we denote the class of all finite structures isomorphic to an (induced) substructure of some $\relsys{A}\in \K$ and call it the {\em age of $\K$}. Similarly, for a structure $\relsys{A}$, the age of $\relsys{A}$, $\Age(\relsys{A})$, is $\Age(\{\relsys{A}\})$.

%The key property of an age of ultrahomogeneous structure is captured by the amalgamation property.

Let $\relsys{A},\relsys{B},\relsys{C}$ be relational structures, $\alpha$ an
embedding of $\relsys{C}$ into $\relsys{A}$, and $\beta$ an embedding of
$\relsys{C}$ into $\relsys{B}$.  An {\em amalgamation of $(\relsys{A},
\relsys{B}, \relsys{C}, \alpha, \beta)$} is any triple
$(\relsys{D},\gamma,\delta)$, where $\relsys{D}$ is a relational structure,
$\gamma$ an embedding $\relsys{A}\to \relsys{D}$ and $\delta$ an embedding
$\relsys{B}\to\relsys{D}$ such that $\gamma\circ\alpha = \delta\circ\beta$.
Less formally, an amalgamation ``glues together'' the structures
$\relsys{A}$ and $\relsys{B}$ into a single substructure of $\relsys{D}$ such
that copies of $\relsys{C}$ coincide. Class $\K$ is have {\em amalgamation property}
if for $\relsys{A},\relsys{B},\relsys{C}\in \K$ and $\alpha$ an embedding of
$\relsys{C}$ into $\relsys{A}$, $\beta$ an embedding of $\relsys{C}$ into
$\relsys{B}$, there exists $(\relsys{D},\gamma,\delta), \relsys{D}\in \K$, that is an amalgamation of
$(\relsys{A}, \relsys{B}, \relsys{C}, \alpha, \beta)$.

A class $\K$ of finite relational structures is called an {\em amalgamation
class} if it is hereditary (i.e. for every $\relsys{A}\in \K$ and induced
substructure $\relsys{B}$ of $\relsys{A}$ we have $\relsys{B}\in \K$), closed
under isomorphism, has countably many mutually non-isomorphic structures and
has the amalgamation property.  The following classical result establishes the
connection in between amalgamation classes and ultrahomogeneous structures.

\begin{theorem}[\Fraisse{} Theorem~\cite{Fraisse:1953}]
\label{fraissethm}
(a) A class $\K$ of finite structures is the age of a countable ultrahomogeneous structure $\relsys{H}$ if and only if $\K$ is an amalgamation class. 
(b) If the conditions of (a) are satisfied then the structure $\relsys{H}$ is unique up to isomorphism. 
\end{theorem}

%Having $n$ fixed, we seek for structural properties of age $\K$ implying the existence of structure $\relsys{U}$.

By \Fraisse{} Theorem the amalgamation property can be seen as the critical property of age $\K$ such that there exists structure $\relsys{U}$ with $\Age(\relsys{U})=\K$ satisfying $\rc(\relsys{A})=0$. Being inspired with this approach we could, for a fixed age $\K$, ask for bounds of relational complexity for structures $\relsys{U}$ having $\Age(\relsys{U}) = \K$ or alternatively having $n$ fixed, we seek for structural properties of age $\K$ implying the existence of structure $\relsys{U}$ with $\Age(\relsys{U}) = \K$ and $\rc(\relsys{U}) = n$.

%Relational complexity is not interesting for rigid structures (with trivial
%automorphism group), where it is always 1. Such a structure exists for
%almost every age.  To get meaningful lower bounds on structures
%of a given age we restrict our attention to $\omega$-categorical
%structures.  Recall that countably infinite structure is called {\em
%$\omega$-categorical} if all countable models of its first order theory are
%isomorphic. Equivalently the $\omega$-categorical structures can be
%characterized as structures whose automorphism group has only finitely many
%orbits on $n$-tuples, for every $n$, and thus also there are only finitely many
%invariant relations of arity $n$ (see \cite{Hodges:1993}). 
Relational complexity is not interesting for rigid structures (with trivial
automorphism group), where it is always 1. Such a structure exists for
almost every age.  We thus restrict our attention to $\omega$-categorical
structures.  Recall 
that countably infinite structure is called {\em
$\omega$-categorical} if all countable models of its first order theory are
isomorphic. Equivalently the $\omega$-categorical structures can be
characterized as structures whose automorphism group has only finitely many
%that structure is $\omega$-categorical if and only if it has only finitely many
orbits on $n$-tuples, for every $n$, and thus also there are only finitely many
invariant relations of arity $n$, see \cite{Hodges:1993}. 

Moreover countable $\omega$-categorical structure $\relsys{U}$ contains as an
induced substructure every countable structure $\relsys{A}$,
$\Age(\relsys{A})\subseteq \Age(\relsys{U})$, see \cite{Cameron:1992}.  We say that
$\relsys{U}$ is {\em universal} for the class of all structures of age at most
$\Age(\relsys{A})$ (also called structures {\em younger} than $\relsys{U}$).

 %Moreover countable $\omega$-categorical
%structure $\relsys{U}$ contains as an induced substructure every countable
%structure $\relsys{A}$, $\Age(\relsys{A})\subseteq \Age(\relsys{U})$, \cite{Cameron:1992}.  We say
%that $\relsys{U}$ is {\em universal} for class of all structures of age at most
%$\Age(\relsys{A})$ (also called structures {\em younger} than $\relsys{U}$).

There is no 1-1 correspondence in between $\omega$-categorical structures and
their ages.  We will demonstrate this on the class of bipartite graphs.

Consider class $\K$ of lifts of finite bipartite graphs with one partition
distinguished by an extended unary relation.  $\K$ is an amalgamation class and
thus there is an (up to isomorphism unique) infinite ultrahomogeneous lift
$\relsys{X}, \Age(\relsys{X})=\K$.
Now consider a bipartite graph $\relsys{B}_2$ that is shadow of $\relsys{X}$.
Graph $\relsys{B}_2$ is not ultrahomogeneous, it is however universal bipartite
graph (because $\relsys{X}$ is universal for bipartite graphs with one
partition distinguished).  All vertices of $\relsys{B}_2$ have infinite degree
and $\relsys{B}_2$ is connected.  Age of $\relsys{B}_2$ is the class of all
finite bipartite graphs.  Every countable bipartite graph can be found as an
induced subgraph of $\relsys{B}_2$ (it is an universal bipartite graph). 
Because $\relsys{B}_2$ is vertex transitive, there are no non-trivial invariant
unary relations.  It is
possible to turn $\relsys{B}_2$ into a ultrahomogeneous lift by aid of 2 invariant
binary relations. First relation connect every two vertices within the same
partition. Second relation connect every pair of vertices in different
partitions. Consequently $\rc(\relsys{B}_2)=2$ while $\lc(\relsys{B}_2)=1$.

Denote by $L$ the set of all vertices of $\relsys{X}$ in the partition
distinguished by the unary relation.  Now consider graph $\relsys{B}_1$
created from $\relsys{X}$ by adding a new vertex of degree 1 for every $v\in L$ 
along with an edge connecting both vertices.  The automorphism of $\relsys{B}_1$ necessarily maps
vertices of degree 1 to vertices of degree 1 and thus can not swap the
partitions.  It is thus possible to turn $\relsys{B}_1$ into a ultrahomogeneous
structure with 2 unary relations and $\rc(\relsys{B}_1)=\lc(\relsys{B}_1)=1$.
The age is however unchanged.

Finally it is possible to construct, for given $n>2$, a connected bipartite graph
$\relsys{B}_n$, $\rc(\relsys{B}_n)=n$. Take any connected bipartite graph $\relsys{A}_n$, $\rc(\relsys{A}_n)=n$. (Such a graph can be constructed by techniques of Section \ref{sec:large}; non-bipartite graphs can be turned into bipartite by subdividing every edge by an vertex). Construct $\relsys{B}_n$ as a disjoint union of $\relsys{B}_n$ and $\relsys{A}_n$ with one vertex unified.
%Put $\relsys{B}_n=\relsys{B}_2+\relsys{A}_n$.

It follows that an $\omega$-categorical relational structure with age
consisting of all finite bipartite graph can have relational complexity
anywhere in between 1 and infinity. This is not a sporadic example and
we thus need to add extra restrictions on the structures in consideration.
Among all $\omega$-categorical structures with a given age we can turn
our attention to the ``most ultrahomogeneous like'' in the following sense. Structure $\relsys{A}$ with $\Age(\relsys{A})=\K$ is {\em
existentially complete} if for every structure $\relsys{B}$,
such that $\Age(\relsys{B})=\K$ and the identity mapping (of $A$) is an embedding $\relsys{A}\to\relsys{B}$,
every existential statement $\psi$ which is defined in $\relsys{A}$ and
true in $\relsys{B}$ is also true in $\relsys{A}$.  By \cite{CherlinShelahShi:1999}
for every age $\K$ defined by forbidden monomorphisms with $\omega$-categorical
universal structure there is also up to isomorphism unique $\omega$-categorical, existentially complete, and $\omega$-saturated
universal structure. This in fact holds more generally.
In such cases, for a given age $\K$, the {\em canonical universal structure} of age $\K$
is the unique $\omega$-categorical, existentially complete, and $\omega$-saturated
structure $\relsys{U}$ such that $\Age(\relsys{U})=\K$.

Given an age $\K$ we can thus ask:
\begin{itemize}
 \item[I.] What is the minimal relational complexity of an $\omega$-categorical
  structure $\relsys{U}$ such that $\Age(\relsys{U})=\K$?
 \item[II.] What is the relational complexity of the canonical universal structure of age $\K$?
\end{itemize}

We consider universal structures for class $\Forb(\F)$ where $\F$ is a family of
connected structures.
 $\Forb(\F)$ denotes
the class of all structures $\relsys{A}$ for which there is no homomorphism
$\relsys{F}\to \relsys{A}$, $\relsys{F}\in \F$.
Classes $\Forb(\F)$ are among the most natural ones where the
existence of a universal structure is guaranteed for every finite $\F$, see \cite{CherlinShelahShi:1999}. 
For such $\F$ we can fully answer the questions
above.

For a structure $\relsys{A}=(A,(\rel{A}{i},i\in I))$, the {\em Gaifman graph} (in combinatorics often called {\em 2-section}) is the graph $G_\relsys{A}$ with vertices $A$ and all those edges which are a subset of a tuple of a relation of $\relsys{A}$, i.e., $G=(A,E),$ where ${x,y}\in E$ if and only if $x\neq y$ and there exists a tuple $\vec{v}\in \rel{A}{i}$, $i\in I$, such that $x,y\in \vec{v}$

For a structure $\relsys{A}$ and a subset of its vertices $B\subseteq A$, the
 {\em neighborhood} of set $B$ is the set of all
vertices of $A\setminus B$ connected in $G_\relsys{A}$ by an edge
to a vertex of $B$.
 We denote by $G_\relsys{A}\setminus B$ the graph
created from  $G_\relsys{A}$ by removing the vertices in $B$.

\begin{figure}[ht!]
\centerline{\includegraphics[width=6cm]{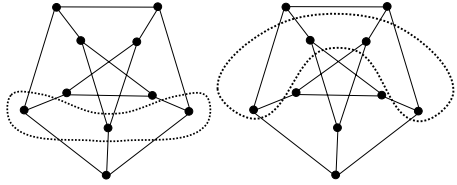}}
\caption{Two minimal g-separating g-cuts of the Petersen graph.}
\label{fig:petersoni}
\end{figure}

A {\em g-cut} in $\relsys{A}$ is a subset $C$ of $A$ that is a vertex cut of $G_\relsys{A}$.
%$G_\relsys{A}$ is disconnected by removing the set $C$. That is, there are
%vertices $u,v\in A\setminus C$ that belong to the same connected component of
%$G_\relsys{A}$ but to different connected components of
%$G_\relsys{A}\setminus C$. 
A g-cut $C$ is {\em minimal g-separating} in $\relsys{A}$ if there exists
structures $\relsys{A}_1\neq \relsys{A}_2$ induced by $\relsys{A}$ on two connected components of $G_\relsys{A}\setminus C$
such that
$C$ is the intersection of the neighborhood of $A_1$ and the neighborhood of $A_2$ in $\relsys{A}$.

All minimal separating cuts of the Petersen graph (up to isomorphisms) are depicted in Figure \ref{fig:petersoni}.

A family of structures is called {\em minimal}  if and only if
all structures in $\F$ are cores and there is no homomorphism between two
structures in $\F$.
\begin{theorem}
\label{mainthm}
Let $\F$ be a finite minimal family of finite connected relational structures and
$\relsys{U}$ an $\omega$-categorical universal structure for $\Forb(\F)$. Denote by $n$ the size of the largest minimal g-separating g-cut in $\F$. Then
(a) $\rc(\relsys{U})\geq n$;
(b) if $\relsys{U}$ is the canonical universal structure for $\Forb(\F)$, then $\rc(\relsys{U})=n$.
\end{theorem}
In the rest of the paper we show the upper bounds and lower bounds given by
Theorem \ref{mainthm}. We prove more general statements in fact.

\subsection {Upper bounds on relational complexity}

It appears that relational complexity is closely related to the homogenization method of
constructing universal structures as used in \cite{Covington:1990}.  
The main result of \cite{Covington:1990} is in fact a variant of \Fraisse{}
Theorem with the amalgamation reduced to so-called local
failure of amalgamation.

%The {\em homogenization} is a process of turning given structure to ultrahomogeneous one by adding
%a relations.  For this we adopt the following terms.
%We call $\relsys{X}$ a {\em lift} of $\relsys{A}$ and $\relsys{A}$ is called
%the {\em shadow} of $\relsys{X}$.  Note that a lift is also in the
%model-theoretic setting called an {\em expansion} (as we are expanding our
%relational language) and a shadow a {\em reduct}.

The {\em Amalgamation failure} of a given age $\K$ is a triple
$(\relsys{A},\relsys{B},\relsys{C})$ such that $\relsys{A},\relsys{B}, \relsys{C} \in \K$,
the identity mapping (of $C$) is an embedding $\relsys{C}\to\relsys{A}$ and $\relsys{C}\to\relsys{B}$, and
there is no amalgamation of $\relsys{A}$ and $\relsys{B}$ over $\relsys{C}$  in $\K$. (i.e.~$(\relsys{A},\relsys{B},\relsys{C})$ shows that $\K$ has no amalgamation property).
Amalgamation failure is {\em minimal} if there is no another amalgamation failure
$(\relsys{A}',\relsys{B}',\relsys{C}')$ such that identity mappings are embeddings $\relsys{A}'\to \relsys{A}$, $\relsys{B}'\to\relsys{B}$ and $\relsys{C}'\to\relsys{C}$.

\begin{theorem}
\label{thm:upperbound2}
Let $\relsys{U}$ be the canonical universal structure for age $\K$
and $S$ the set of isomorphism types of minimal amalgamation failures of
$\relsys{U}$. If $S$ is finite then $\rc(\relsys{U})$ and $\lc(\relsys{U})$ is bounded from above by
the largest size of $\relsys{C}$ such that  $(\relsys{A}, \relsys{B},
\relsys{C})\in S$.
\end{theorem}

\begin{proof}
Given an age $\K$ and set $S$, 
\cite{Covington:1990} provide construction of ultrahomogeneous lift $\relsys{X}$ such that its
shadow $\relsys{U}$ is universal for the class of structures of age $\K$.
Moreover this lift is constructed using relations invariant in 
the automorphism group of $\relsys{U}$ (Lemma~2.7 in \cite{Covington:1990}) and their
arities correspond to the sizes of $\relsys{C}$ such that  $(\relsys{A},
\relsys{B}, \relsys{C})\in S$.  Because $\relsys{U}$ is an shadow of
$\relsys{X}$, it is $\omega$-categorical.  The existential
completeness and $\omega$-saturation follows directly from the 
construction.
\end{proof}

\begin{figure}[ht!]
\centerline{\includegraphics[width=8cm]{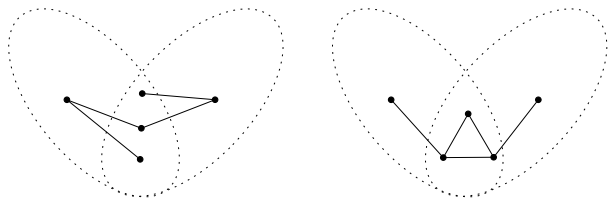}}
\caption{Amalgamation failures of the class of graphs not containing an induced path of length 3.}
\label{amalfailure}
\end{figure}

\paragraph{\bf Example.}
Determining the minimal set of amalgamation failures of a given age can be
difficult. \cite{Covington:1990} shows only one non-trivial example, the cographs (i.e. graphs without induced path on 4 vertices.) The minimal amalgamation failures are
depicted in Figure \ref{amalfailure}. This class has only 2
minimal failures both with 3 vertices in $\relsys{C}$. The resulting lift
use relation of arity 3. This is in contrast with the fact that all finite
cographs have complexity at most 2 (Proposition~\ref{prop:cographs}).

In the special case of $\K=\Age(\Forb(\F))$ one can prove a stronger result.
\begin{theorem}
\label{thm:upperbound}
Let $\F$ be a (finite or infinite) family of connected structures 
such that there exists $\relsys{U}$, the canonical universal structure for $\Age(\Forb(\F))$.
Then $\rc(\relsys{U})$ and $\lc(\relsys{U})$ is bounded from above by the size of the largest minimal
g-separating g-cut in $\F$.
\end{theorem}
\begin{proof}[Proof (sketch).]
This is a consequence of \cite{BubickaNesetril:toapear}.
The main construction in \cite{BubickaNesetril:toapear} builds lift of the class $\Forb(\F)$
similar to one in \cite{Covington:1990}.
The new relations corresponds to individual g-separating g-cuts of structures in $\F$.
This immediately give the bound on lift complexity (stated as \cite{BubickaNesetril:toapear},
Theorem 3.1).

There is however important difference in between \cite{Covington:1990} and
\cite{BubickaNesetril:toapear}. The lifts produced in \cite{BubickaNesetril:toapear} are built in fully
algorithmic way and they avoid the existential completeness argument.  In our
case this is an complication; indeed the shadows of lifts introduced in
\cite{BubickaNesetril:toapear} are not existentially complete.  To fully show
Theorem \ref{thm:upperbound} one needs to re-prove \cite{BubickaNesetril:toapear} with existential
completeness argument. This will appear elsewhere.
\end{proof}

\paragraph{{\bf Examples.}}
Theorem \ref{thm:upperbound} can be easily applied to many families $\F$.  For example:

\begin{enumerate}
\item Let $\F$ be family of relational trees and $\relsys{U}$ the canonical universal structure for $\Forb(\F)$ (if it exists).
By Theorem \ref{thm:upperbound} $\rc(\relsys{U})\leq 1$. In fact $\relsys{U}$ can be seen
as a ``blown up'' core of a homomorphism dual $\relsys{D}$ (given by \cite{NesetrilTardif:2000} even for some infinite families $\F$) where each vertex is replaced by infinitely many
vertices and each edge by a random bipartite graph. In this case the bound given by Theorem \ref{thm:upperbound2} is not tight even for $\F$ consisting
of an oriented path on 4 vertices.

\item Let $\F_{C_n}$ contain a single odd graph cycle on $n$ vertices. The relational complexity the canonical universal structure for $\Forb(\F_{C_n})$ is at most 2.

\item Let $\F_{odd}$ be the class of all odd graph cycles. The canonical universal structure for $\Forb(\F_{odd})$ is the random bipartite graph $\relsys{B}_2$. By Theorem \ref{thm:upperbound}
$\rc(\relsys{B})\leq 2$.
\end{enumerate}

\subsection {Lower bounds on relational complexity}
\label{lowerbound}

We obtain the following bound:
\begin{theorem}
\label{prop1}
Let $\F$ be a finite minimal family of finite connected structures and
$\relsys{U}$ an $\omega$-categorical universal structure for $\Forb(\F)$.
Then $\rc(\relsys{U})$ and $\lc(\relsys{U})$ is bounded from bellow by the
size of largest minimal g-separating g-cut in $\F$.
\end{theorem}
We use of the following result proved by a special Ramsey-type construction.
This is not a technical finesse but this is in a way necessary. It has been shown by \cite{N1,N4} that Ramsey classes are related to ultrahomogeneous structures. This connection has been elaborated in the context of topological dynamics in \cite{KPT}.

\begin{theorem}[\cite{BubickaBubak}]
\label{aritythm}
Let $\F$ be a finite minimal family of finite connected relational structures
and $\K$ a lift of class $\Forb(\F)$ adding finitely many new relations of
arity at most $r$. If $\K$ contains ultrahomogeneous lift $\relsys{U}$ that is universal for $\K$ then
the size of minimal g-separating g-cuts of $\relsys{F}\in \F$ is bounded by $r$.
\end{theorem}
The Theorem \ref{prop1} follows directly. Fix
$\omega$-categorical $\relsys{U}$ universal for $\Forb(\F)$.  By
$\omega$-categoricity of $\relsys{U}$ we know that
$\Inv_k(\relsys{A})$ is finite and we apply Theorem~\ref{aritythm}.  

\paragraph{{\bf Examples.}}
\begin{enumerate}
\item Let $\F_P$ consist of the Petersen graph and let $\relsys{U}$ be the canonical structure
for $\Forb(\F_P)$. Then $\rc(\relsys{U})=\lc(\relsys{U})=4$. Recall that the minimal g-separating g-cuts
of the Petersen graph are shown in Figure~\ref{fig:petersoni}.
\item The complexity of an $\omega$-categorical graph universal for $\Forb(\F_{C_n})$,
$n\geq 5$, (of graphs without odd cycles of length at most $n$) is at least 2.
Combining with Theorem \ref{thm:upperbound} we know that relational complexity
of the canonical universal structure for the class $\Forb(\F_{C_n})$ is 2. On
the other hand, however, this does not hold for the class  $\F_{odd}$. We have
already shown that the canonical universal graph for the class of all bipartite
graphs have relational complexity 2 and lift complexity 1.
Finiteness and minimality assumptions are thus both essential in Theorem~\ref{prop1}.
\end{enumerate}

\section{Concluding remarks}
We have shown examples of ages $\K$ where the relational complexity of the canonical
universal structure for $\K$ agrees with its lift complexity. This is the case
of $\Forb(\F)$ classes where $\F$ is finite family of finite connected
structures.  In the case of bipartite graphs the complexities disagree by one
(relational complexity is 2, while lift complexity is 1). It is also possible
to construct examples where relational complexity is finite and lift
complexity is 1.  One such example is the bow-tie free graph considered in
\cite{CherlinShelahShi:1999}.  To see this fact however an careful analysis of
its structure is needed that is out of scope of this paper.

Cameron and Ne\v set\v ril~\cite{CameronNesetril:2006} recently introduced
concept of homomorphism-homogeneous relational structures.  It would be
interesting to study the topic of this paper in the context of this type of
homogeneity.

%%%%%%%%%%%%%%%%%%%%%%%%%%%%%%%%%%%%%%%%%%%%%%%%%%%%%%%%%%%%%%%%%%%%%%%

% 
%\bibliographystyle{unsrt}
%\bibliography{homhom-L-col-graphs}

\begin{thebibliography}{99}


\bibitem{Cameron:1992} 
Cameron, P.~J.: 
\textit{The age of a relational structure},
Directions in Infinite Graph Theory and Combinatorics (ed. R.~Diestel), Topic in Discrete Math. 3, 49--67,
North-Holland, Amsterdam, 1992

\bibitem{CameronNesetril:2006}
Cameron, P.~J. and Ne{\v{s}}et{\v{r}}il, J.:
\textit{Homomorphism-homogeneous relational structures},
Combinatorics Probability and Computing \textbf{15}(1-2) (2006), 91--103.

\bibitem{Cherlin:1998} 
Cherlin, G.: 
\textit{The classification of countable homogeneous directed graphs and countable homogeneous $n$-tournaments},
Mem. Amer. Math. Soc., 612 (1998), 1 -- 161.

\bibitem{Cherlin:2000}
Cherlin, G.:
\textit{Sporadic homogeneous structures},
The Gelfand Mathematical Seminars (1994-1999), 15--48, 
Birkh\"{a}user Boston, 2000.

\bibitem{Cherlin:2011}
Cherlin, G.:
\textit{Two problems on homogeneous structures, revisited},
in Model Theoretic Methods in Finite Combinatorics, M. Grohe and J.A. Makowsky eds.,
Contemporary Mathematics, 558, American Mathematical Society, 2011.

\bibitem{CherlinMartinSaracino:1996}
Cherlin, G. and Martin, G. and Saracino, D.:
\textit{Arities of permutation groups: Wreath products and $k$-sets},
J. Combinatorial Theory, Ser. A \textbf{74} (1996), 249--286.

\bibitem{CherlinShelahShi:1999}
Cherlin,G.~L. and  Shelah, S. and  Shi N.: 
\textit{Universal Graphs with Forbidden Subgraphs and Algebraic Closure},
Advances in Applied Mathematics 22 (1999), 454--491.

\bibitem{Covington:1990}
Covington, J: 
\textit{Homogenizable Relational Structures},
Illinois J. Mathematics, 34(4) (1990), 731--743.

\bibitem{Fraisse:1953}
Fra\"{\i}ss\' e , R.:
\textit{Sur certains relations qui g\'{e}n\'{e}ralisent l'ordre des nombres rationnels},
C.R. Acad. Sci. Paris \textbf{237} (1953), 540--542.

\bibitem{Gardiner:1976}
Gariner, A.:
\textit{Homogeneous graphs},
J. Combin. Theory Ser. B \textbf{20(1)} (1976), 94--102.

\bibitem{Hodges:1993}
Hodges, W.:
\textit{Model Theory},
Cambridge University Press, Cambridge, 1993.

\bibitem{BubickaNesetril:toapear}
Hubi\v{c}ka, J. and  Ne\v set\v ril, J.:
\textit{Universal structures with forbidden homomorphisms},
to appear in J. V\" a\" an\" anen Festschrift, Ontos.

\bibitem{BubickaBubak} 
Hubi\v{c}ka, J. and  Ne\v set\v ril, J.:
\textit{Homomorphism and embedding universal structures for restricted classes}, arXiv:0909.4939.

\bibitem{JenkinsonSeidelTruss:2012}
Jenkinson, T. and  Seidel, D. and  Truss, J. K.: 
\textit{Countable homogeneous multipartite graphs},
European J. Combin., 33 (2012), 82--109.

\bibitem{KPT} 
Kechris, A. S. and  Pestov, V. G. and  Todor\v cevi\v c, S.: 
\textit{Fra\"{\i}ss\' e Limits, Ramsey Theory, and Topological Dynamics of Automorphism Groups}. Geom. Funct. Anal., 15 (2005), 106--189.

\bibitem{KnightLachlan:1985}
Knight, J. and  Lachlan A.:
\textit{Shrinking, stretching, and codes for homogeneous structures}.
Classification Theory, J. Baldwin ed., LNM 1292, 
Sringer, New York, 1985.

\bibitem{LachlanWoodrow:1980} 
Lachlan, A. H. and  Woodrow, A. H.: 
\textit{Countable ultrahomogeneous graphs},
Trans. Amer. Math. Soc., 262(1) (1992), 51--94.

\bibitem{NesetrilTardif:2000} 
Ne{\v{s}}et{\v{r}}il, J. and  Claude Tardif, C.: 
\textit{Duality Theorems for Finite Structures (Characterising Gaps and Good Characterisations).},
J. Comb. Theory, Ser. B, 80(1) (2000), 80--97.

\bibitem{MatousekNesetril:1998}
Matou\v sek, J and  Ne{\v{s}}et{\v{r}}il, J.:
\textit{Invitation to discrete mathematics},
Oxford University Press, Oxford, 1998.

\bibitem{N1} Ne\v set\v ril, J.: \textit{For graphs there are only four types of hereditary Ramsey Classes}, J. Combin. Theory B, 46(2) (1989), 127--132.
\bibitem{N4} Ne\v set\v ril, J.: \textit{Ramsey Classes and Homogeneous Structures}, Combinatorics, Probablity and Computing (2005) 14, 171--189.%Ramsey Theory. In: Handbook of Combinatorics (ed. R. L. Graham, M. Gr\" otschel, L. Lov\' asz), Elsevier (1995), 1331--140

\bibitem{Rose:2011}
Rose, S. E.:
\textit{Classification of Countable Homogeneous $2$-Graphs},
PhD thesis, University of Leeds, 2011.


\end{thebibliography}

\end{document}